\documentclass [12pt,a4paper] {amsart}

\usepackage[utf8]{inputenc}
\usepackage[T1]{fontenc}
\usepackage{lmodern}
\usepackage{microtype}

\usepackage {amssymb}
\usepackage {amsmath}
\usepackage {amsthm}
\usepackage{mathtools}
\usepackage{xfrac}

\usepackage{url}
 
\usepackage[titletoc,toc,title]{appendix}

\usepackage{enumitem}
\usepackage{hyperref}
\hypersetup{
    colorlinks=true,
    linkcolor=blue,
    filecolor=magenta,      
    urlcolor=cyan,
    citecolor=magenta
}
\usepackage{cleveref}

\usepackage{tikz}
\usetikzlibrary{arrows}

\DeclareMathOperator {\im} {Im}
\DeclareMathOperator {\pr} {pr}

\DeclareMathOperator {\seq} {\subseteq}
\DeclareMathOperator {\C} {\mathbb{C}}
\DeclareMathOperator {\R} {\mathbb{R}}

\DeclareMathOperator {\Z} {\mathbb{Z}}

\newcommand{\s}{\mathbb{S}}

\newcommand{\zbar}{\ensuremath{\mathbf{z}}}
\newcommand{\wbar}{\ensuremath{\mathbf{w}}}

\newcommand{\re}{\mathrm{Re}}
\newcommand{\Cx}{\mathbb{C}^{\times}}

\newcommand{\pbar}{{\ensuremath{\bar{p}}}}
\newcommand{\qbar}{{\ensuremath{\bar{q}}}}

\theoremstyle {plain}
\newtheorem {theorem}{Theorem}[section]

\newtheorem {lemma} [theorem] {Lemma}
\newtheorem{fact}[theorem]{Fact}
\newtheorem {proposition} [theorem] {Proposition}
\newtheorem {corollary} [theorem] {Corollary}

\newtheorem{innercustomthm}{Theorem}
\newenvironment{customthm}[1]
  {\renewcommand\theinnercustomthm{#1}\innercustomthm}
  {\endinnercustomthm}

\newtheorem{innercustomprop}{Proposition}
\newenvironment{customprop}[1]
  {\renewcommand\theinnercustomprop{#1}\innercustomprop}
  {\endinnercustomprop}

\theoremstyle {definition}

\newtheorem {definition}[theorem]{Definition}
\newtheorem{example} [theorem] {Example}

\theoremstyle {remark}
\newtheorem {remark} [theorem] {Remark}

\calclayout

\title{Complex solutions of polynomial equations on the unit circle}
\author{Vahagn Aslanyan}

\email{Vahagn.Aslanyan@manchester.ac.uk}
\address{Department of Mathematics, University of Manchester, Manchester, UK}

\begin{document}

\vspace*{-3cm}

\thanks{This work was supported by EPSRC Open Fellowship EP/X009823/1. For the purpose of open access, the author has applied a Creative Commons Attribution (CC BY) licence to any Author Accepted Manuscript version arising from this submission.}



\maketitle

\begin{abstract}
    We explore systems of polynomial equations where we seek complex solutions with absolute value 1. Geometrically, this amounts to understanding intersections of algebraic varieties with tori -- Cartesian powers of the unit circle. We study the properties of varieties in which this intersection is Zariski dense, give a criterion for Zariski density and use it to show that the problem is decidable. This problem is a ``continuous'' analogue of the Manin-Mumford conjecture for the multiplicative group of complex numbers, however, the results are very different from Manin-Mumford. 
    
    While the results of the paper appear to be new, the proofs are quite elementary. This is an expository article aiming to introduce some classical mathematical topics to a general audience. We also list some exercises and problems at the end for the curious reader to further explore these topics.
\end{abstract}

\section{Introduction}

Finding solutions of (systems of) equations, or understanding when there is one, is a fundamental problem in mathematics. When asking whether a given equation has a solution, one must always specify the space where solutions are sought. For instance, it does not make sense to ask whether the equation $x^2+1=0$ has a solution. We can ask if this equation has a real solution, and the answer would be ``No''. But if we ask whether it has a complex solution then the answer would be ``Yes''. 

Let us focus on polynomial equations or systems thereof, i.e. (systems of) equations of the form $p(z_1,\ldots,z_n)=0$ where $p(z_1,\ldots,z_n)$ is a polynomial with complex coefficients. The Fundamental Theorem of Algebra states that any non-constant polynomial in a single variable has a complex zero. Hilbert's Nullstellensatz extends this theorem providing a criterion for the existence of complex solutions of systems of polynomial equations in several variables.

Often we are interested in \textit{special} types of solutions. For instance, in number theory a classical problem asks when a polynomial equation has a rational or integral solution. For instance, one can prove easily that the equation $x^2+y^2=1$ has infinitely many rational solutions, e.g. $\left( \frac{3}{5}, \frac{4}{5}\right)$. On the other hand, by a celebrated theorem of Andrew Wiles, for $n>2$ the equation $x^n+y^n=1$ has no non-trivial rational solutions with the trivial solutions being $(0, \pm 1)$ and $(\pm 1, 0)$ (this is known as Fermat's last theorem). In these examples the special solutions are the rational ones. Polynomial equations where one is interested in rational or integral solutions are known as \textit{Diophantine equations}, named after the third century Hellenistic mathematician Diophantus of Alexandria. In general it is not possible to decide if a given Diophantine equation has a rational solution, but in many cases it is possible to tell whether an equation has infinitely many such solutions or not. A renowned result in this direction is Faltings's theorem (also known as Mordell's conjecture) stating that under a certain geometric condition, a polynomial equation in two variables can have only finitely many rational solutions. 

Often we are interested in other types of special solutions of arithmetic importance, such as solutions in roots of unity. By extension, such equations are also referred to as Diophantine equations. As an example, let us consider an equation $p(z,w)=0$ where $p$ is a polynomial with complex coefficients and ask when it has infinitely many solutions with both $z$ and $w$ complex roots of unity. For instance, the equation $z^2w=1$ has infinitely many such solutions, for if $\zeta$ is a root of unity then so is $\zeta^{-2}$ and the pair $(\zeta, \zeta^{-2})$ is a solution to the equation. On the other hand, the equation $z+w=2$ has only one solution in roots of unity, namely, $z=w=1$. To see this observe that if $|z|=|w|=1$ then by the triangle inequality $|z+w|\leq |z|+|w|=2$ and equality holds if and only if $w=c z$ for some non-negative real number $c$. But then $c$ must be $1$, hence $z=w=1$. It turns out that more generally only ``multiplicative'' equations can have infinitely many solutions in roots of unity.

\begin{fact}[Ihara, Serre, Tate, {\cite{Zannier-book-unlikely,Pila-ZP-book}}]\label{fact: Ihare-Serre-Tate}
        Let $f$ be an irreducible polynomial. Assume the equation $f(x,y)=0$ has infinitely many solutions $(\xi, \eta)$ whose coordinates are roots of unity. Then up to multiplication by a constant $f$ is of the form $x^my^n-\zeta$ where $m, n \in \Z$ and $\zeta$ is a root of unity. 
\end{fact}

Note that when $m$ or $n$ is negative then $x^my^n-\zeta$ is not a polynomial. However, it is a \textit{Laurent polynomial} where negative powers of the indeterminates may occur. When we work in the multiplicative group of non-zero complex numbers, it is often more convenient to work with Laurent polynomials which we will do. Nevertheless, one can always multiply a Laurent polynomial by a suitable monomial to get rid of all negative powers of the indeterminates and work with classical polynomials.

There is also a generalisation of Fact \ref{fact: Ihare-Serre-Tate} to systems of equations in several variables. It is convenient to state it in geometric terms, where instead of an equation or a system of equations we consider a geometric object -- an \textit{algebraic variety} defined by these equations. Examples of algebraic varieties are curves and surfaces. Then a solution of the system of equations under consideration is just a point on the corresponding variety. In this language, Fact \ref{fact: Ihare-Serre-Tate} states that a plane curve contains infinitely many points with both coordinates roots of unity if and only if the curve is defined by an equation of the form $x^my^n=\zeta$.

In higher dimensions we have an algebraic variety $V\seq \C^n$ for some $n\geq 2$ and we want to understand the set of points in $V$ all of whose coordinates are roots of unity. Here the question is not whether $V$ contains infinitely many such points; that is too weak of a condition to characterise $V$. For example, the variety in $\C^3$ defined by the equation $xy+z=2$ contains the infinite set $\{ (\zeta, \zeta^{-1},1): \zeta \text{ is a root of unity} \}$ but it is not of a multiplicative form. Instead, one asks whether $V$ contains a \textit{Zariski dense} subset of points whose coordinates are roots of unity. Recall that a subset of a variety $V$ is said to be \textit{Zariski dense} in $V$ if $V$ is the smallest algebraic variety containing that set. In other words, a subset is Zariski dense in $V$ if every polynomial that vanishes at all points of that subset also vanishes on the whole $V$.

\begin{fact}[Manin-Mumford conjecture; Laurent {\cite{laurent,Pila-ZP-book}}]\label{fact: Manin-Mumford}
    Let $V\seq \C^n$ be an irreducible algebraic variety containing a Zariski dense set of torsion points of $(\Cx)^n$. Then $V$ is defined by (one or more) equations of the form $z_1^{a_1}\cdots z_n^{a_n} = \zeta$ where $a_k \in \Z$ and $\zeta$ is a root of unity.
\end{fact}

The philosophy of Facts \ref{fact: Ihare-Serre-Tate} and \ref{fact: Manin-Mumford} is that an algebraic variety contains few \textit{special points} (points whose coordinates are roots of unity) unless there is a trivial reason for it to contain many such points, in which case it contains as many special points as it possibly can.




\subsection{Main results}

When we showed that $(1,1)$ is the only special solution of the equation $z+w=2$, the only property of roots of unity that we used was that they have absolute value 1. It turns out that in some cases this is the only property of roots of unity that we need to prove finiteness of special solutions. So it makes sense to study solutions of polynomial equations where each coordinate of the solution has absolute value 1. In other words, instead of studying intersections of algebraic varieties with powers of the set of all roots of unity, we study intersections with powers of the unit circle $\s_1$, also known as tori.\footnote{Note that $\s_1$ is the Euclidean closure of the set of all roots of unity in $\C$ and is the maximal compact subgroup of the multiplicative group $(\Cx, \cdot)$.} This question was studied in \cite{corvaja-masser-zannier}, where the authors explore intersections of varieties with the maximal compact subgroup $\Gamma$ of a commutative algebraic group of dimension 2. They prove that for some such groups the conclusion of Manin-Mumford holds under the weaker assumption that the variety merely contains a Zariski dense subset of points from $\Gamma$. They remark that this fails for the multiplicative group of non-zero complex numbers, give some examples to illustrate this and obtain some basic results in dimension 2. In this paper we provide a thorough investigation of this problem in an arbitrary dimension. Some of our ideas and proofs are based on those of  \cite{corvaja-masser-zannier}, while others are somewhat different, albeit still elementary.

\textbf{Caveat.} Our results are written in the language of basic algebraic geometry, e.g. varieties, dimension, Zariski density, etc. We present the necessary preliminaries in the next section. However, a reader who is not familiar with these notions and finds them hard to grasp may assume for simplicity that $n=2$ throughout the paper. That means varieties are just curves defined by polynomial equations of the form $p(z_1,z_2)=0$, the dimension of a curve is 1, and a Zariski dense subset in a curve is simply an infinite subset.

Our first main result is the following criterion for Zariski density.

\begin{theorem}\label{thm: V = graph alg map}
    Let $V \seq \C^n$ be an irreducible algebraic variety of dimension $d$. Assume the projection of $V$ to the first $d$ coordinates has dimension $d$. Then $V\cap \s_1^n$ is Zariski dense in $V$ if and only if there is a holomorphic map $F=(f_1,\ldots,f_{n-d}): U \to \C^{n-d}$ defined on an open set $U \seq (\Cx)^d$ such that 
    \begin{itemize}
        \item $V$ contains the graph of $f$, and 
        \item for all $(z_1,\ldots,z_d)\in \bar{U}\cap U^{-1}$ and for each coordinate function $f_k$ of $F$ we have $\overline{f_k(\bar{z}_1,\ldots,\bar{z}_d)}\cdot f_k(z_1^{-1},\ldots,z_d^{-1})=1$.
    \end{itemize}    
    Moreover, the map $F$ is necessarily algebraic and we may assume $U \cap \s_1^d \neq \emptyset$.
\end{theorem}

Note that when $\dim V = d$, the projection of $V$ to some $d$ coordinates has dimension $d$. By renumbering the coordinates we may assume these are the first $d$ coordinates. So this assumption in the statement of the above theorem does not restrict the generality.

This theorem can be used to establish a dichotomy result: the intersection of an algebraic variety $V\seq \C^n$ with $\s_1^n$ is either small (not Zariski dense) or it has the highest possible dimension. 

\begin{corollary}\label{cor: intersection is large}
    Let $V \seq \C^n$ be an algebraic variety such that $V \cap \s_1^n$ is Zariski dense in $V$. Then the real dimension of $V \cap \s_1^n$ is equal to the complex dimension of $V$.
\end{corollary}

As a consequence of Theorem \ref{thm: V = graph alg map} we show that there is an algorithm which, given an algebraic variety $V\seq \C^n$, decides whether it has a Zariski dense intersection with $\s_1^n$. One says in this case that the problem of Zariski density of $V\cap \s_1^n$ in $V$ is \textit{decidable}.

Our next result is about some geometric properties of $V$ related to the Zariski density of $V \cap \s_1^n$ in $V$.

\begin{proposition}\label{prop: V=V* and odd degree}
    Let $V\seq \C^n$ be an irreducible algebraic variety of dimension $d$.
    \begin{itemize}
        \item If $V \cap \s_1^n$ is Zariski dense in $V$ then $\bar{V} = V^{-1}$ where $\, \bar\,$ and $^{-1}$ are applied to each point of $V$ coordinate-wisely.

        \item If $\bar{V} = V^{-1}$, and the projection of $V$ to some $d$ coordinates is dominant and of odd degree, then $V\cap \s_1^n$ is Zariski dense in $V$.
    \end{itemize}
\end{proposition}

Thus, the first is a necessary condition and the second is a sufficient criterion for Zariski density. Unfortunately we are unable to provide a nice geometric necessary and sufficient condition, and we argue in Remark \ref{rem: nec and suff} that such a criterion does not exist.

However, in some special cases we are able to give an explicit characterisation of varieties intersecting tori in a Zariski dense subset. This can be done when the variety is the graph of a rational map due to the following statement.

\begin{proposition}\label{prop: rational map fixing S_1}
        Let $p(z_1,\ldots,z_n)\in \C[z_1,\ldots,z_n]$. Then the rational function \[ r(z_1,\ldots,z_n):=\frac{\overline{p(\overline{z_1},\ldots, \overline{z_n})}}{p(z_1^{-1},\ldots,z_n^{-1})} \] maps $\s_1^n$ to $\s_1$. Moreover, all rational functions mapping $\s_1^n$ to $\s_1$ are of this form, up to multiplication by a monomial $\beta z_1^{t_1}\cdots z_n^{t_n}$ where $\beta \in \s_1$ and $t_k\in \Z$ for all $k$. 
        
        In particular, any rational function in a single variable fixing $\s_1$ setwise is of the form
    \[  \prod_{k=1}^m \frac{z-\alpha_k}{1-\overline{\alpha_k}z} \] where the $\alpha_k$'s are arbitrary complex numbers.
\end{proposition}

At the end we explain how for some simple plane curves $C\seq \C^2$ all points of the intersection $C\cap \s_1^2$ can be found. We also suggest some problems in the last section to encourage an in-depth investigation of the phenomena presented in the paper.

\section{Preliminaries}\label{sec: prelim}

In this section we set out the necessary preliminaries used throughout the paper.

\subsection{Notation and conventions}

We begin by introducing some notation and conventions that we need later.

\begin{itemize}
    \item Indeterminates of polynomials, variables of functions, as well as coordinates, will be denoted by $z_1,z_2,...$ and $w_1,w_2,...$.

    \item Tuples will be denoted by boldface letters, e.g. $\zbar = (z_1,\ldots,z_n)$ and $\zbar = (z_1,\ldots,z_n)$. The length of the tuple will be clear from the context.

    \item For a tuple $\zbar \in \C^n$ we let $\bar{\zbar} := (\bar{z_1},\ldots,\bar{z_n})$ where $\,\bar{}\,$ denotes complex conjugation. Similarly, for $\zbar \in (\Cx)^n$, where $\Cx = \C\setminus \{ 0 \}$, we write $\zbar^{-1} := (z_1^{-1}, \ldots, z_n^{-1})$. For a set $U \seq \C^n$ (or $U\seq (\Cx)^n$) we write $\bar{U}$ and $U^{-1}$ for the images of $U$ under the maps $\, \bar{}\,$ and $^{-1}$ as defined above.

    \item We let $^*: \Cx \to \Cx$ denote the composition of $\, \bar\,$ and $^{-1}$. We extend it to tuples and sets as above.

    \item For a tuple $\zbar=(z_1,\ldots,z_n)$ and a tuple of integers $\mathbf{t}=(t_1,\ldots,t_n)$ we write $\zbar^{\mathbf{t}}:= \prod_{k=1}^n z_k^{t_k}$.

    \item For a function $f: U \seq \C^n \to \C$ we define a function $\bar{f}: \bar{U} \to \C$ by $\bar{f}(\zbar) = \overline{f(\bar{\zbar})}$. If $f$ is holomorphic, then so is $\bar{f}$. For instance, if $f$ is a polynomial then $\bar{f}$ is the polynomial obtained from $f$ by conjugating its coefficients.

    \item Algebraic varieties will always be defined over the complex numbers and will be identified with the sets of their complex points. Thus, when we write $V\seq \C^n$ is an algebraic variety we mean that $V$ is a subset of $\C^n$ defined by polynomial equations.

    \item $\s_1$ denotes the unit circle, i.e. $\s_1:=\{ z\in \C: |z|=1 \}$.

    \item Given a tuple $\bar{k}=(k_1,\ldots,k_d)\in \{ 1,2,\ldots, n \}^d$ we denote by $\pr_{\bar{k}}:\C^n \to \C^d$ the projection map $(z_1,\ldots,z_n)\mapsto (z_{k_1},\ldots,z_{k_d})$.
\end{itemize}

\subsection{Algebraic varieties and Zariski density}

\begin{definition} Let $n\geq 1$ be an integer.
\begin{itemize}
    \item     An \textit{algebraic variety} or a \textit{Zariski closed set} in $\C^n$ is the set of common zeroes of some finite list of polynomials in $n$ variables. For instance, the set $\{ (z_1,z_2,z_3): z_2=z_1^2, 2iz_1^4+\pi z_2^3+ez_3^5=0 \}$ is an algebraic variety in $\C^3$.

    \item An algebraic variety $V$ is called \textit{irreducible} if it cannot be written as a union of two other varieties each of which is properly contained in (i.e. is contained in and is not equal to) $V$.

    \item A subset $U$ of an algebraic variety $V$ is said to be \textit{Zariski dense} in $V$ if $V$ is the smallest algebraic variety containing $V$; in other words, any polynomial that vanishes at each point of $U$ must also vanish at each point of $V$. For instance, any infinite subset of $\C$ is Zariski dense in $\C$.

    
\end{itemize}
\end{definition}

\begin{example}
    The set $\s_1^n$ is Zariski dense in $\C^n$. Let us prove this for $n=2$. Assume $\s_1^2$ is not Zariski dense in $\C^2$, i.e. there is a non-zero polynomial $p(z_1,z_2)$ vanishing on $\s_1^2$. Then for some $t\in \s_1$ the polynomial $p(x,t)\in \C[x]$ is non-zero and vanishes on all of $\s_1$. This is not possible, for $p(x,t)$ has only finitely many zeroes. 
\end{example}

\subsection{Complex and real dimensions}

\begin{definition}

\begin{itemize}
    \item[] 
    \item     The dimension of a variety $V\seq \C^n$ is the largest number $d$ for which the projection of $V$ to some $d$ coordinates contains a non-empty open subset of $\C^n$. This is also referred to as the complex dimension of $V$.

    \item Given a subset $V$ of $\R^n$ cut out by polynomial equations with real coefficients, its real dimension is the largest number $d$ for which the projection of $V$ to some $d$ coordinates contains a non-empty open subset of $\R^n$.
\end{itemize}   
\end{definition}

In this papers algebraic varieties are assumed to be defined over the complex numbers and are subsets of $\C^n$, so often write $\dim V$ instead of $\dim_{\C}V$ and refer to it as the dimension of $V$ (without specifying the word ``complex''). Every algebraic variety $V\seq\C^n$ can be thought of as a subset of $\R^{2n}$ defined by replacing the complex variables and coefficients of polynomials by their real and imaginary parts. For instance, the equation $z_2=z_1+1+2i$ defines a curve in $\C^2$ (i.e. it has complex dimension 1) and corresponds to the set $\{ (x_1,y_1,x_2,y_2): x_2=x_1+1, y_2=y_1+2 \}$ in $\R^4$ which has real dimension 2. In general, for a variety $V\seq \C^n$ we have $\dim_{\R}V = 2\dim_{\C}V$. There are subsets of $\R^n$ which are defined by real polynomial equations but do not correspond to a variety in $\C^n$. For instance, the unit circle in $\R^2$ is defined by the equation $x^2+y^2=1$. This corresponds to the equation $z\bar{z}=1$ which is not an algebraic equation any more. In such situations it makes sense to talk about real dimension (in this example it is 1), but not complex dimension.

\subsection{Dominant projections and degree}

\begin{definition}
    Let $V \seq \C^n$ be an algebraic variety and let $\pi : V \to \C^d$ be the projection map to some $d$ coordinates. 

    \begin{itemize}
        \item We say $\pi: V \to \C^d$ is \textit{dominant} if $\dim \pi(V) = d$. If $d=\dim V$ then there is a dominant projection to some $d$ coordinates.

        \item If $d = \dim V$ and $\pi: V \to \C^d$ is dominant then there is a Zariski closed set $W\subsetneq \C^d$ with $\dim W < d$ such that for all $\wbar\in \C^d\setminus W$ the fibres $\pi^{-1}(\wbar)$ are finite and of the same size. This finite number is called the degree of $\pi$, written $\deg \pi$. If $\deg \pi = m$ then ``generically'' $\pi$ is $m$-to-$1$ (i.e. one point has $m$ preimages). 
    \end{itemize}
\end{definition}

\begin{example}
    Consider the curve $w^2=z$ in $\C^2$ and let $\pi$ be the projection to the $z$-coordinate. Then the projection is dominant and every non-zero point $w$ has exactly two preimages under $\pi$, that is, $\deg \pi = 2$. Note that $0$ has only one preimage (of multiplicity two though).
\end{example}

The reader is referred to any algebraic geometry textbook (e.g. \cite{Shafarevich}) for further details on algebraic varieties and other related concepts.

\subsection{Algebraic maps}

Some polynomial equations are simpler than others. For instance, the equation $z^2w=z^3+1$ allows one to express $w$ in terms of $z$ as the rational function $\frac{z^3+1}{z^2}$, which comes in useful in many situations. This cannot be done for the equation $w^2=z^3-1$, so it is somewhat harder to study. However, it is still possible to express $w$ in terms of $z$ provided that we are happy to consider algebraic functions which are more general than rational functions. Thus, the equation $w^2=z^3-1$ may be thought of as a pair of equations $w=\sqrt{z^3-1}$ and $w=-\sqrt{z^3-1}$. But we must first understand where and how the function $\sqrt{z^3+1}$ can be defined. While this is not a difficult task, it is a subtle one, so we should be careful. Doing this in higher dimensions makes the problem even more complicated, so we give a brief account of algebraic functions next. The reader is referred to \cite[\S 3]{Aslanyan-Kirby-Mantova} for details. See also \cite[pp. 284--308]{ahlfors} for a detailed exposition of the construction of algebraic functions in a single variable.

\begin{example}
    Consider the equation $w^2=z$. Expressing $w$ in terms of $z$ we get $\pm \sqrt{z}$. First, we note that $\sqrt{z}$ is locally well defined around any non-zero point by the implicit function theorem. The partial derivative $\frac{\partial(w^2-z)}{\partial w}$ vanishes at $z=0,w=0$, so the implicit function theorem cannot be applied there. The point $(0,0)$ is exceptional in the sense that for any non-zero $z$ the equation $w^2=z$ has two solutions in $w$, while for $z=0$ there is only one solution (of multiplicity 2). One says that $(0,0)$ is a \textit{ramification point}. Thus, if we define $\sqrt{z}$ and $-\sqrt{z}$ outside $0$, then at $0$ these two branches merge together and violate the nice properties of these functions (e.g. continuity). This may lead one to think that two continuous branches of $\sqrt{z}$ can be defined on the puncture plane $\Cx$. Unfortunately, this is not the case. To understand this, let us suppose we work in a small disc around the point $1$. Since the function $z$ does not vanish there, we can write it as $\exp(u)$ where $u = \log z$ is a branch of logarithm. Then the two branches of $\sqrt{z}$ would be $\exp(\sfrac{u}{2})$ and $-\exp(\sfrac{u}{2})$. Let us work with the first one. This function can be analytically continued to the whole complex plane. However, $\exp(u+2\pi i) = \exp(u)$ while $\exp(\frac{u+2\pi i}{2})= -\exp(\frac{u}{2})$. Thus, the same value of $z=\exp(u)$ would necessarily be mapped to two values by the function $\sqrt{z}$ if it could be defined in a punctured neighbourhood of zero. The reason of this failure is that a punctured disc is not a simply connected domain. Indeed, it is known that $\log z$ can be defined only in simply connected subsets of $\Cx$ and similarly $\sqrt{z}$ can be defined on such subsets too. Thus, we need to make a \textit{branch cut}, that is, remove a subset of $\Cx$ making it simply connected. For instance, removing the negative real half-line would result in a simply connected set where two branches of $\sqrt{z}$ can be defined.
\end{example}

\begin{example}
    Now let us consider the example we started with, $w^2=z^3-1$. Here the ramification locus is when $z^3=1$, i.e. $z$ is a third root of unity (of which there are three). So if we remove a real line from the complex plane containing all third roots of unity then the function $\sqrt{z^3-1}$ can be defined on such a set.
\end{example}

More generally, given a polynomial equation $p(z,w)=0$ of degree $d$ in $w$, there is a branching locus, namely, the finite set of points where the partial derivative $\frac{\partial p}{\partial y}$ vanishes. This is the set where the value of $z$ corresponds to $<d$ values of $w$. Then on any simply connected set not containing any of these points $d$ holomorphic functions $f_1(z),\ldots,f_d(z)$ can be defined which satisfy the following equality:
\[ p(z,w) = c\prod_{k=1}^d (w-f_k(z)). \]

In general, the following holds (see \cite[Proposition ??]{Aslanyan-Kirby-Mantova}).

\begin{fact}\label{fact: alg maps}
    Let $V \seq \C^n$ be an irreducible algebraic variety of dimension $d$. Let $\bar{k}:=(k_1,\ldots,k_d)$ be a tuple such that the projection map $\pr_{\bar{k}} : V \to \C^d$ is dominant and of degree $m$.\footnote{This means the projection map is $m$-to-$1$ outside a proper Zariski closed subset.} Then there is a Zariski closed subset $W\seq \C^d$ of dimension $<d$ such that for any simply connected set $U \seq \C^d$ with $U\cap W = \emptyset$ there are $m$ analytic maps $F_1,\ldots,F_m : U \to \C^{n-d}$ such that
    \[ \forall \zbar \in U\, \forall \wbar \in \C^{n-d}\, [ (\zbar,\wbar)\in V \text{ iff } \wbar=F_k(\zbar) \text{ for some } k ].  \]
\end{fact}

The coordinate functions of the maps $F_k$ are algebraic as per the following definition: a holomorphic function $f:U\seq \C^n \to \C$ is said to be algebraic if there is a non-zero polynomial $p(\zbar,w)\in \C[z_1,\ldots,z_n, w]$ such that $p(\zbar,f(\zbar))=0$ for all $\zbar \in U$.

\section{Proofs of the main results}

In this section we prove our main results, which we restate here for the convenience of the reader. The formulations given below use the notation and nomenclature introduced in \S\ref{sec: prelim} and are slightly different from, but trivially equivalent to, the statements in the introduction.

\subsection{Properties of varieties with a Zariski dense intersection with $\s_1^n$}



\begin{customthm}{\ref{thm: V = graph alg map}}
    Let $V \seq \C^n$ be an irreducible algebraic variety of dimension $d$. Assume the projection of $V$ to the first $d$ coordinates is dominant. Then $V\cap \s_1^n$ is Zariski dense in $V$ if and only if there is a holomorphic map $F=(f_1,\ldots,f_{n-d}): U \to \C^{n-d}$ defined on an open set $U \seq (\Cx)^d$ such that 
    \begin{itemize}
        \item $V$ contains the graph of $f$, and 
        \item for all $\zbar\in U\cap U^*$ and for each $k$ we have $\overline{f_k(\zbar)}\cdot f_k(\zbar^*)=1$.
    \end{itemize}    
    Moreover, the map $F$ is necessarily algebraic and we may assume $U \cap \s_1^d \neq \emptyset$.
\end{customthm}

\begin{proof}
 Since the projection of $V$ to the first $d$ coordinates is dominant, by Fact \ref{fact: alg maps} there are algebraic maps $F_1,\ldots,F_m: U \to \C^{n-d}$ defined on a simply connected domain  $U \seq \C^d$ (avoiding the ramification locus) such that for every $(\zbar,\wbar) \in U \times \C^{n-d}$ we have \[ (\zbar,\wbar) \in V \text{ iff } \wbar = F_k(\zbar) \text{ for some } k. \]

Now we specify how we choose $U$. Since the ramification locus $R \seq \C^d$ has dimension $<d$, the real dimension of $R \cap \s_1^d$ must be less than $d$. Therefore, we can choose $U\seq \C^d\setminus R$ simply connected such that $U \cap \s_1^d = \s_1^d \setminus R$.


Now assume $V\cap \s_1^n$ is Zariski dense in $V$. Then the following set must also be Zariski dense in $V$: \[ \{ (\zbar, \wbar)\in (U\times \C^{n-d})\cap V \cap \s_1^n \} = \bigcup_{i=1}^m\{ (\zbar, F_i(\zbar)): \zbar \in U \cap \s_1^d \text{ and } F_i(\zbar)\in \s_1^{n-d} \}.  \]
    Therefore, for some $i=i_0$ the set $\{ (\zbar, F_{i_0}(\zbar)): \zbar \in U \cap \s_1^d \text{ and } F_{i_0}(\zbar)\in \s_1^{n-d} \}$ is Zariski dense in $V$. For ease of notation, we write $F:=F_{i_0}$. Then the set 
    \[ \{ \zbar\in U \cap \s_1^d: F(\zbar)\in \s_1^{n-d} \}\] is Zariski dense in $\C^d$.

    Now let $\bar{F}(\zbar):=\overline{F(\bar{\zbar})} : \bar{U} \to \C^{n-d}$. Then $\bar{F}$ is an algebraic map defined on $\bar{U}$. To see this, we write $F$ and $\bar{F}$ in coordinates as $F=(f_1,\ldots,f_{n-d})$ and $\bar{F}=(\bar{f}_1,\ldots,\bar{f}_{n-d})$. Then if $f_i$ satisfies an equation $p_i(\zbar,f_i(\zbar))=0$ where $p_i(\zbar,w_i)$ is an irreducible polynomial, then $\bar{f}_i$ satisfies the equation $\overline{p_i}(\zbar, \bar{f}_i(\zbar))=0$ where $\overline{p_i}(\zbar, w_i) := \overline{p_i(\bar{\zbar}, \bar{w_i})}$ is also a polynomial. 
    
    When $\zbar \in U\cap \s_1^d$, we have $\zbar^{-1} = \bar{\zbar}$ and $\bar{f}_k(\zbar^{-1}) = \overline{f_k(\zbar)}$, and so $f_k(\zbar)\bar{f}_k(\zbar^{-1})=|f_k(\zbar)|^2$.
    Thus, if $\zbar \in U\cap \s_1^d$ then $F(\zbar)\in \s_1^{n-d}$ if and only if for all $k$ we have $f_k(\zbar)\bar{f}_k(\zbar^{-1})=1$. Therefore, $f_k(\zbar)\bar{f}_k(\zbar^{-1})-1$ vanishes on a Zariski dense subset of $\C^d$. We claim that it must identically vanish on $U \cap U^*$. This follows immediately from the definition of Zariski density when $h_k(\zbar):=f_k(\zbar)\bar{f}_k(\zbar^{-1})-1$ is a polynomial. In general, $h_k$ is an algebraic function satisfying an equation $q_k(\zbar, h_k(\zbar))=0$ where $q_k(\zbar, w_k)$ is an irreducible polynomial. This means that $q_k(\zbar,0)$ vanishes on a Zariski dense subset of $\C^d$, hence it is identically zero. Hence, $q_k(\zbar,w_k)$ is divisible by $w_k$ and, since it is irreducible, it must coincide with a constant multiple of $w_k$. In other words $h_k=0$, which is what we wanted to prove.
    
Conversely, if there is an $F$ as in the statement of the theorem, then it must agree with one of the $F_k$'s on $U$. In particular, the domain $U$ can be extended to contain an open subset of $\s_1^d$ so we assume this is the case. Then the identity \[ \overline{f_k(\zbar^*)}\cdot f_k(\zbar)=1 \] holds on $U\cap U^*$. In particular, for any $\zbar \in U \cap \s_1^d$ we get $F(\zbar)\in \s_1^{n-d}$ which shows that $\dim_{\R}(V\cap \s_1^n)=d$ and $V\cap \s_1^n$ is Zariski dense in $V$.
\end{proof}

Corollary \ref{cor: intersection is large} follows immediately from this proof.

\begin{remark}
   It can be shown that under the assumptions of Theorem \ref{thm: V = graph alg map} the projection of $V\cap \s_1^n$ to the first $d$ coordinates can have a $d$-dimensional complement in $\s_1^d$. See Exercise \ref{prob: proper arc} in \S\ref{sec: problems}.
\end{remark}

Proposition \ref{prop: V=V* and odd degree} and its proof are direct extensions of the $n=2$ case considered in \cite[p. 228]{corvaja-masser-zannier}.

\begin{customprop}{\ref{prop: V=V* and odd degree}}
    Let $V\seq \C^n$ be an irreducible algebraic variety of dimension $d$.
    \begin{itemize}
        \item If $V \cap \s_1^n$ is Zariski dense in $V$ then $V = V^*$.

        \item If $V = V^*$, and the projection of $V$ to some $d$ coordinates is dominant and of odd degree, then $V\cap \s_1^n$ is Zariski dense in $V$.
    \end{itemize}
\end{customprop}

\begin{proof}
    Since the restriction of $^*$ to $\s_1^n$ is the identity map, $V^* \cap \s_1^n = (V\cap \s_1^n)^* = V\cap \s_1^n$. Taking Zariski closure of both sides we get $V^*=V$.

    Now suppose the projection $\pi$ to the first $d$ coordinates is dominant and of odd degree. For any $\zbar \in \s_1^d \cap \pi(V)$ the set $V_{\zbar}:=\{ \wbar\in \C^{n-d}: (\zbar, \wbar)\in V \}$ is closed under $^*$. For $\zbar$ outside a proper Zariski closed subset of $\C^d$, the set $V_{\zbar}$ is finite and its size is equal to $\deg \pi$ which is assumed to be odd. Therefore, for such $\zbar$ the map $^*: V_{\zbar}\to V_{\zbar}$ must have a fixed point $\wbar$, which then is necessarily from $\s_1^{n-d}$. Then $(\zbar,\wbar)\in V \cap \s_1^n$ and the set of all such points has real dimension $d$ and so is Zariski dense in $V$.
\end{proof}

\subsection{Varieties which are graphs of rational maps}

\begin{lemma}\label{lem: alg map fixing s_1}
    Given an algebraic function $g: U \seq \C^n \to \C$ with $U \cap \s_1^n \neq \emptyset$, the function $f(\zbar) := \frac{g(\zbar)}{\bar{g}(\zbar^{-1})}$, defined on $\hat{U}:= U \cap U^*$, maps $\hat{U}\cap \s_1^n$ to $\s_1$. Conversely, for any algebraic function $f:U\seq \C^n \to \C$ with $U\cap \s_1^n\neq \emptyset$ which maps $U\cap \s_1^n$ to $\s_1$, we have $f(\zbar)^2 = \frac{f(\zbar)}{\bar{f}(\zbar^{-1})}$ on $U \cap U^*$.    
\end{lemma}
\begin{proof}
    First, observe that since $^*$ fixes $\s_1$ pointwise, $U\cap \s_1^n = U \cap U^* \cap \s_1^n$, so these two intersections are non-empty at the same time.
    
    For the first part of the lemma, straightforward calculations show that
    \[ f(\zbar)\cdot \overline{f(\zbar^*)} = f(\zbar)\cdot \bar{f}(\zbar^{-1}) = \frac{g(\zbar)}{\bar{g}(\zbar^{-1})} \cdot \frac{\bar{g}(\zbar^{-1})}{g(\zbar)} = 1.  \]

    In particular, if $\zbar \in \hat{U}\cap \s_1^n$ then $f(\zbar) \cdot \overline{f(\zbar)} = 1$ hence $f(\zbar)\in \s_1$.

    For the converse, since $f$ maps $U \cap \s_1^n$ to $\s_1$, for all $\zbar \in U \cap \s_1^n$ we have $f(\zbar)\cdot \bar{f}(\zbar^{-1}) = f(\zbar) \cdot \overline{f(\zbar)} = 1$. Since $f(\zbar)\cdot \bar{f}(\zbar^{-1})$ is algebraic on $U \cap U^*$ and is equal to $1$ on $U \cap U^*\cap \s_1^n$, it must be equal to $1$ on all of $U \cap U^*$. Then
    $ \frac{f(\zbar)}{\bar{f}(\zbar^{-1})} = f(\zbar)^2$ on $U \cap U^*$.
\end{proof}

This lemma gives a recipe for constructing varieties with a Zariski dense intersection with an appropriate torus. For instance, if $g(\zbar)$ is an algebraic function then the minimal polynomial of $\frac{g(\zbar)}{\bar{g}(\zbar^{-1})}$ over $\C(\zbar)$ can be multiplied through by a common denominator and turned into a polynomial in $\C[\zbar, w]$ which then defines a variety in $\C^{n+1}$ with a Zariski dense intersection with $\s_1^{n+1}$. The lemma also shows that all varieties containing a Zariski dense subset from a torus can be obtained in this way. Indeed, the conclusion that $f(\zbar)^2 = \frac{f(\zbar)}{\bar{f}(\zbar^{-1})}$ on $\hat{U} := U \cap U^*$ can be written as $f(\zbar) = \frac{g(\zbar)}{\bar{g}(\zbar^{-1})}$ with $g(\zbar) = \sqrt{f(\zbar)}$ on a suitable simply connected subset of $\hat{U}$ that avoids the zeroes and poles of $f$.

Now we can use this lemma to prove Proposition \ref{prop: rational map fixing S_1}.


\begin{customprop}{ \ref{prop: rational map fixing S_1}}
        Let $p(z_1,\ldots,z_n)\in \C[z_1,\ldots,z_n]$. Then the rational function $r(\zbar):=\frac{p(\zbar)}{\pbar(\zbar^{-1})}$ maps $\s_1^n$ to $\s_1$. Moreover, all rational functions mapping $\s_1^n$ to $\s_1$ are of this form, up to multiplication by a monomial $\beta z_1^{t_1}\cdots z_n^{t_n}$ where $\beta \in \s_1$ and $t_k\in \Z$ for all $k$. 
        
        In particular, any rational function in a single variable fixing $\s_1$ setwise is of the form
    \[  \prod_{k=1}^m \frac{z-\alpha_k}{1-\overline{\alpha_k}z} \] where the $\alpha_k$'s are arbitrary complex numbers.
\end{customprop}

\begin{proof}

By Lemma \ref{lem: alg map fixing s_1} the map $\frac{p(\zbar)}{\bar{p}(\zbar^{-1})}$ sends $\s_1^n$ to $\s_1$, hence so does the map $\frac{\overline{p(\bar{\zbar})}}{p(\zbar^{-1})}$.

    Now suppose $r(\zbar)=\frac{p(\zbar)}{q(\zbar)}$ maps $\s_1^n$ to $\s_1$, and assume $p$ and $q$ are coprime in the ring of polynomials $\C[\zbar]$. Then by Lemma \ref{lem: alg map fixing s_1} we have $r(\zbar)\cdot \bar{r}(\zbar^{-1})=1$.
    Hence
    \[ p(\zbar)\cdot \pbar(\zbar^{-1}) = q(\zbar)\cdot \qbar(\zbar^{-1}).  \]
    We consider this equality in the ring of Laurent polynomials $\C[\zbar, \zbar^{-1}]$ which is a unique factorisation domain. Since $p$ and $q$ are coprime, there must be Laurent polynomials $f$ and $g$ such that
    \[  p(\zbar) = f(\zbar) \qbar(\zbar^{-1}) \text{ and } q(\zbar) = g(\zbar) \pbar(\zbar^{-1}). \]
    From these we get
    \[ q(\zbar) = g(\zbar) \bar{f}(\zbar^{-1}) q(\zbar). \]
    Now we can conclude that $f(\zbar)$ and $g(\zbar)$ are units in the ring of Laurent polynomials, that is, constant multiples of monomials. Therefore, $q(\zbar) = \beta \cdot \zbar^{\mathbf{t}}\pbar(\zbar^{-1})$ with $\mathbf{t}\in \Z^n$. Since $r$ maps $\s_1^n$ to $\s_1$, and so does the function $\zbar^{\mathbf{t}}\pbar(\zbar^{-1})$, we must have $\beta\in \s_1$. Then we can write
    \[  r(\zbar) = \beta^{-1}\zbar^{-\mathbf{t}} \frac{p(\zbar)}{\pbar(\zbar^{-1})}. \]
    which is of the required form.

    Finally, when $r$ is a rational function of a single variable then decomposing its numerator and denominator into linear factors gives the required form.
\end{proof}

\subsection{Deciding whether $V\cap \s_1^n$ is Zariski dense in $V$}

The characterisation given above does not immediately tell us whether for a given variety $V$ we can understand if $V\cap \s_1^n$ is Zariski dense in $V$. We address this question here.

\begin{proposition}
    There is an algorithm which, given an algebraic variety $V\seq \C^n$, decides if $V\cap \s_1^n$ is Zariski dense in $V$. 
\end{proposition}
\begin{proof}
    The proof is based on simple ideas from the model theory of the field of real numbers. We do not assume any familiarity with that theory and present an informal proof sketch below. The readers who know the basics of model theory should be able to formalise this sketch easily. Those who do not but are interested are referred to \cite[\S 3.3]{Marker}.

    Let $d:=\dim V$ and assume without loss of generality that the coordinates $z_1,\ldots,z_d$ are algebraically independent on $V$. Then by Corollary \ref{cor: intersection is large}, $V\cap \s_1^n$ is Zariski dense in $V$ if and only if the following holds:
    \begin{equation}\label{eq: zariski dense formula}
         \exists \mathbf{u}\in \s_1^d\, \exists \delta > 0\, \forall \zbar \in \s_1^d \,(|\zbar - \mathbf{u}|< \delta \to \exists \wbar \in \s_1^{n-d}\, (\zbar,\wbar)\in V).
    \end{equation}
    If we interpret $\C$ as $\R^2$ then the unit circle $\s_1$ is definable by a first-order\footnote{Here ``first-order'' refers to first-order logic where formulas are expressions made up of variables, constants (such as 0, 1, $e$, $\pi$), functions (such as $+, \cdot$), relations (such as $<$), quantifiers (i.e. $\forall$ and $\exists$), and logical connectives (e.g. combining two formulas by ``and'' or ``or''). An example of a formula is $\exists x (x^2+e = \pi)$ which holds in the field of real numbers for $\pi > e$.} formula, hence \eqref{eq: zariski dense formula} is expressible as a first-order formula. It is know that the theory of the field of real numbers is decidable, i.e. there is an algorithm which decides whether any given first-order formula holds in $\R$. In particular, that algorithm can be applied to decide whether for a given variety $V$ the formula \eqref{eq: zariski dense formula} holds.
\end{proof}

\begin{remark}\label{rem: nec and suff}
    It would be natural to ask whether there is an explicit geometric characterisation of those algebraic varieties that intersect tori in a Zariski dense subset. There seems to be no such characterisation. To understand this, first note that this is a problem about the field of real numbers. More precisely, we need to understand when a certain system of polynomial equations has a real solution, or a large set of real solutions. While we showed above that this problem is decidable, i.e. for each individual variety we can understand whether it contains a Zariski dense subset of points from the appropriate torus, there seems to be no criterion for this. This is analogous to the following (simplified) problem: is there a criterion for understanding whether a polynomial of a single variable has a real root. Given an arbitrary polynomial $p(x)\in \R[x]$, there is a finite set of polynomial equalities and inequalities which are satisfied by the coefficients of $p$ if and only if it has a real root. So again, the problem is decidable. However, there is no general criterion for understanding whether any given polynomial has a real root. There are some sufficient conditions though. For instance, any polynomial of odd degree with real coefficients has a real root. This is analogous to Proposition \ref{prop: V=V* and odd degree}, and a careful reader would also notice that these facts are both based on the same idea.
\end{remark}

\section{Finding solutions in simple cases}

Consider a plane curve defined by an equation $p(z,w)=0$. As before, we are looking for solutions $(z,w)\in \s_1^2$. We show how to find all solutions in the simple case when the polynomial $p$ has at most three terms. 

First, if $p$ has only one term, everything is trivial. When $p$ has two terms then the equation can be written in the form $z^aw^b=\alpha$ where $\alpha\in \C$ is a constant, $a,b$ are integers. This equation has a solution in $\s_1^2$ if and only if $|\alpha|=1$. If this is the case, then for any $\zeta \in \s_1$ the pair $(\zeta, (\alpha\zeta^{-a})^{\sfrac{1}{b}})$ is a solution (for any choice of the root), and these are all the solutions. 

Now assume $p$ has three terms, i.e. we want to solve the equation 
\begin{equation}\label{eq: x^ay^b+...}
    \xi z^{a_1}w^{b_1}+\eta z^{a_2}w^{b_2}+\zeta z^{a_3}w^{b_3}=0
\end{equation}
in $\s_1^2$. We start with a special case and consider an equation of the form
\begin{equation}\label{eq: x+y=beta}
  z+\alpha w = \beta  
\end{equation}
where $\alpha, \beta \neq 0$ are complex numbers. Observe that if $(z,w)$ satisfies this equation and $|z|=|w|=1$ then $|\beta - z| = |\alpha|$. So $z$ is on the circle of radius $|\alpha|$ centred at $\beta$. Denote this circle by $S(\beta, |\alpha|)$ (the red circle in Figure \ref{fig: circles}) On the other hand, $z$ is on $\s_1$ (blue circle in Figure \ref{fig: circles}). Thus, $z$ must be an intersection point of these two circles. There are three possibilities.

\begin{enumerate}
    \item $|\alpha|+1 < |\beta|$. Then $S(\beta, |\alpha|)\cap \s_1 = \emptyset$ and \eqref{eq: x+y=beta} has no solutions.

    \item $|\alpha|+1 = |\beta|$. Then $S(\beta, |\alpha|)\cap \s_1 $ consists of a single point, giving the value of $z$. This then uniquely determines the value of $w$. Thus \eqref{eq: x+y=beta} has one solution.

    \item $|\alpha|+1 > |\beta|$. Then $S(\beta, |\alpha|)\cap \s_1 $ consists of two points which give two possibilities for the value of $z$, each of which uniquely determines $w$. Hence, \eqref{eq: x+y=beta} has two solutions.
\end{enumerate}

    \begin{figure}\label{fig: circles}
\begin{tikzpicture}[baseline = 0, scale = 2]

\draw[thick,->] (-1.5,0) -- (3.2,0) node[anchor=north] {$\re$};
\draw[thick,->] (0,-1.25) -- (0,1.9) node[anchor=east] {$\im$};

\draw[thick,blue] (0,0) circle (1cm);

\draw[thick,red] (1.5,0.5) circle (1.2cm);

\draw[brown!50] (0,0) circle (1.2cm);

    \draw (0.92,-0.1) node {\tiny {1}};
    
  \draw (-1,0) node [below right] {\tiny {-1}};
  
      \draw (0,0) node [below left] {\tiny {0}};
      \filldraw(0,1) circle  node [below left] {\scriptsize $i$};

      \filldraw (1.5,0.5) circle (0.5pt) node [right] {\tiny {$\beta$}};

      \filldraw (0.38,0.92) circle (0.5pt) node [below left] {\scriptsize {$z$}};

      \draw (1.5, 0.5) -- (1.5-0.38, 0.5-0.92) --  (0,0) -- (0.38,0.92) -- (1.5, 0.5);

       \filldraw (1.5-0.38, 0.5-0.92) circle (0.5pt) node [ right] {\scriptsize {$\alpha w$}};

       \filldraw (-1.15, 0.34) circle (0.5pt) node [ above left] {\tiny {$\alpha$}};

       \filldraw (-0.997, 0.07) circle (0.5pt) node [ above right] {\tiny {$w$}};

\end{tikzpicture}
\caption{Solving the equation $z+\alpha w = \beta$}
\end{figure}
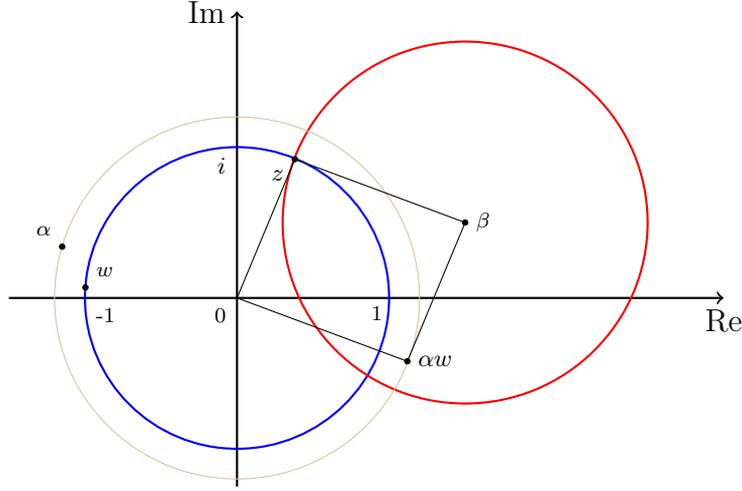 

Now let us turn to equation \eqref{eq: x^ay^b+...}. First we divide by $\zeta z^{a_3}w^{b_3}$ and assume the equation is of the form
\[ z^aw^b + \alpha z^cw^d = \beta. \]
Since $|z|=|w|=1$ implies $|z^aw^b| = |z^cw^d| = 1$, we can apply the above method to find all possible values of $z^aw^b$ and $z^cw^d$. This then is enough to find all possible values of $z$ and $w$. Note that there will be finitely many solutions unless $ad-bc=0$.

\section{Exercises and further problems}\label{sec: problems}

The following exercises and problems will help the reader thoroughly understand the key ideas and arguments in the paper and fill in the gaps in the general picture. Some of the problems are expected to encourage further research around the topics presented here.

\begin{enumerate}

    \item Find all points $z,w\in \s_1$ such that $2z^2w^3+3izw^2=\frac{6}{5}+\frac{23}{5}i$.

    \item \label{prob: C=C* empty} Give an example of a plane curve $C\seq \C^2$ such that $C = C^*$ and $C\cap \s_1^2 = \emptyset$.

    \item Give an example of a rational function mapping $\s_1$ to a proper subset of itself. In other words, give an example of a non-surjective rational function mapping $\s_1$ to itself.

    \item \label{prob: proper arc} Give an example of an irreducible plane curve $C \seq \C^2$ such that $C\cap \s_1^2$ is Zariski dense in $C$ but the projections of $C\cap \s_1^2$ to the first and second coordinates are both proper arcs of $\s_1$.

    \item Show that the set $A:=\left\{ \left(x+i\sqrt{1-x^2}, x^2+i\sqrt{1-x^4}\right)\in \C^2: x\in [0,1] \right\} \seq \s_1^2$ is not Zariski dense in $\C^2$ and find its Zariski closure $C\seq \C^2$. Compare $C\cap \s_1^2$ to $A$.

    \item Let $U\seq \C^n$ be an open subset with $U\cap \s_1^n \neq \emptyset$. Prove that if a holomorphic function $f: U\to \C$ vanishes on $U\cap \s_1^n$ then it vanishes on $U$. \textit{Hint:} Cf. \cite[Proposition 6.10]{Gallinaro-exp-sums-trop}.

    \item \label{prob: anal func} Prove that Lemma \ref{lem: alg map fixing s_1} holds with analytic functions instead of algebraic.

    \item Find all meromorphic functions defined on an open set containing the closed unit disc $\{ z\in \C : |z|\leq 1 \}$ which map $\s_1$ to $\s_1$. \textit{Hint:} Use the maximum modulus principle.

    \item Give an example of a holomorphic function on $\Cx$ which has an essential singularity at $0$ and maps $\s_1$ to $\s_1$. \textit{Hint:} Use Exercise \ref{prob: anal func}.

    \item Show that the set $\left\{ \left( e^{ i t}, e^{i t^2} \right) : t \in [0,1] \right\}\seq \s_1^2$ has real dimension 1 but is Zariski dense in $\C^2$. For any $n\geq 1$ give an example of a subset of $\s_1^n$ of real dimension 1 which is Zariski dense in $\C^n$.

    \item Is it possible to extend Proposition \ref{prop: rational map fixing S_1} to \textit{rational varieties}, i.e. varieties which can be parametrised by rational functions? In particular, is it possible to characterise all curves which can be parametrised as $\{ (f(z), g(z)): z\in \C\setminus P \}$ with $f,g \in \C(z)$, where $P$ is the set of poles of $f$ and $g$, and which contain a Zariski dense subset of $\s_1^2$?
\end{enumerate}

\bibliographystyle{alpha}
\bibliography{ref}

\end{document}